\documentclass[11pt]{article}
\usepackage{xcolor}
\usepackage{amssymb,amsmath,amsthm,bm,citesort,setspace,mathrsfs}
\usepackage{color,units,graphicx,subfigure}
\usepackage[lowtilde]{url}
\usepackage[T1]{fontenc}

\newcommand{\Z}{\mathbf{Z}}

\newcommand{\R}{\mathbf{R}}

\newcommand{\be}{\begin{equation}}
\newcommand{\ee}{\end{equation}}

\newcommand{\lip}{\text{\rm L}_\sigma }

\renewcommand{\P}{\mathrm{P}}
\newcommand{\E}{\mathrm{E}}
\newcommand{\F}{\mathscr{F}}

\renewcommand{\d}{{\rm d}}

\newcommand{\e}{{\rm e}}
\renewcommand{\ge}{\geqslant}
\renewcommand{\le}{\leqslant}
\renewcommand{\geq}{\geqslant}
\renewcommand{\leq}{\leqslant}

\author{Davar Khoshnevisan\\University of Utah
\and Kunwoo Kim\\University of Utah}
\title{Non-linear noise excitation\\and intermittency under high disorder\thanks{
	Research supported in part by the NSF grant DMS-1006903}}
\date{Last update: March 3, 2013}
\newtheorem{stat}{Statement}[section]
\newtheorem{proposition}[stat]{Proposition}

\newtheorem{theorem}[stat]{Theorem}
\newtheorem{lemma}[stat]{Lemma}
\theoremstyle{definition}

\numberwithin{equation}{section}
\begin{document}
\spacing{1.1}
\maketitle
\begin{abstract}
Consider the semilinear heat equation
$\partial_t u = \partial^2_x u + \lambda\sigma(u)\xi$ on the
interval $[0\,,1]$ with Dirichlet zero boundary condition 
and a nice non-random initial function, where 
the forcing $\xi$ is space-time white noise and $\lambda>0$
denotes the level of the noise. We show that, when the solution is intermittent
[that is, when $\inf_z|\sigma(z)/z|>0$], the expected
$L^2$-energy of the solution grows at least as $\exp\{c\lambda^2\}$
and at most as $\exp\{c\lambda^4\}$ as $\lambda\to\infty$. In the
case that the Dirichlet boundary condition is replaced by a
Neumann boundary condition,  we prove that the $L^2$-energy of the solution
is in fact of sharp exponential order $\exp\{c\lambda^4\}$. 
We show also that, for a large family of one-dimensional randomly-forced wave
equations, the energy of the solution grows as $\exp\{c\lambda\}$ as 
$\lambda\to\infty$. Thus, we observe the surprising result that the stochastic
wave equation is, quite typically, significantly less noise-excitable than its parabolic
counterparts.\\


\end{abstract}\newpage

\section{Introduction}
The principal aim of this paper is to study the non-linear effect of noise in 
stochastic partial differential equations whose solutions are ``intermittent.''
Our analysis is motivated in part by a series of computer simulations that we
would like to present first. 

Consider the following stochastic heat equation on the  interval $[0\,,1]$
with homogeneous Dirichlet boundary conditions:
\begin{equation}\label{PAM}\left[\begin{split}
	&\frac{\partial}{\partial t} u_t(x) = \frac{1}{2}\frac{\partial^2}{\partial x^2}
		u_t(x) + \lambda u_t(x) \xi&
		\text{for $0< x< 1$ and $t>0$},\\
	&u_t(0)=u_t(1)=0&\text{for  $t>0$},\\
	&u_0(x) = \sin(\pi x)&\text{for $0<x<1$}.
\end{split}\right.\end{equation}
Here, $\xi$ denotes space-time white noise on $(0\,,\infty)\times[0\,,1]$,
and $\lambda>0$ is an arbitrary parameter that is known as the \emph{level
of the noise}.

The stochastic partial differential equation \eqref{PAM} and its variations 
are sometimes referred to  as \emph{parabolic Anderson models}.
Those are a family of noise-perturbed partial differential equations that arise in
a diverse number of scientific disciplines. For a 
representative sample see Bal\'azs et al \cite{BQS},
Bertini and Cancrini \cite{BC},
Carmona and Molchanov \cite{CM94}, 
Corwin \cite{Corwin},
Cranston and Molchanov \cite{CM},
Cranston, Mountford, and Shiga \cite{CMS,CMS1},
G\"artner and K\"onig \cite{GartnerKonig},
den Hollander \cite{denHollander},
Kardar \cite{Kardar},
Kardar et al \cite{KPZ},
Majda \cite{Majda},
Molchanov \cite{Molch91},
and Zeldovich et al \cite{ZRS},
together with their substantial bibliographies.

Since the initial  function is the principle eigenfunction of the Dirichlet Laplacian
on $[0\,,1]$, it is easy to see that the solution is exactly
$u_t(x)=\sin(\pi x)\exp\{-\pi^2 t/2\}$ when $\lambda=0$; this is the noise-free
case, and a numerical solution is shown is in Figure \ref{fig:lambda0}.

Figures \ref{fig:lambda0.1}, \ref{fig:lambda2}, and \ref{fig:lambda5}  contain
simulations of the solution to \eqref{PAM} for respective 
noise levels $\lambda=0.1$, $\lambda=2$, and $\lambda=5$. 

\begin{figure}\centering
	\includegraphics[width=4in]{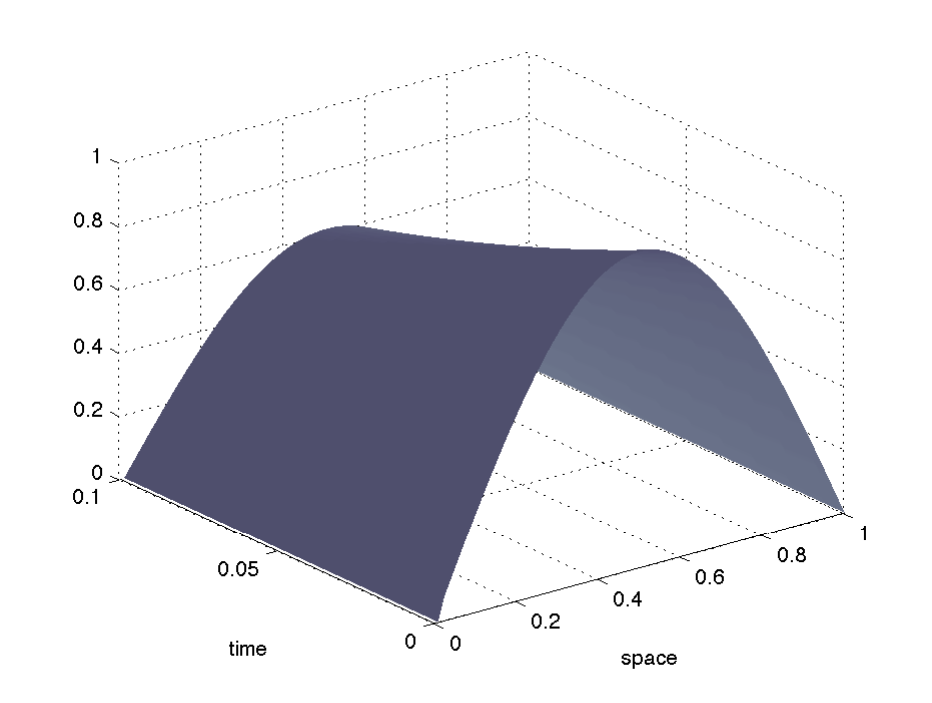}
	\caption{No noise ($\lambda=0.0$,  maximum height $=1.00$)}
		\label{fig:lambda0}
\end{figure}

\begin{figure}[h!]\centering
	\includegraphics[width=4in]{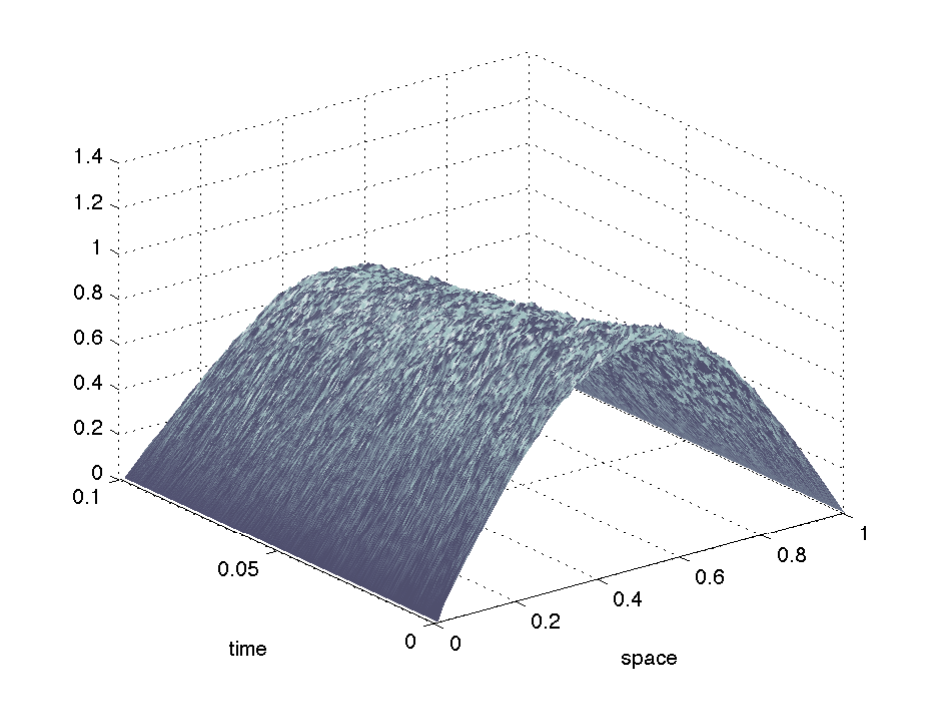}
	\caption{Small noise ($\lambda=0.1$,  maximum height $\approx 1.05$)}
		\label{fig:lambda0.1}
\end{figure}

\begin{figure}[h!]\centering
	\includegraphics[width=4in]{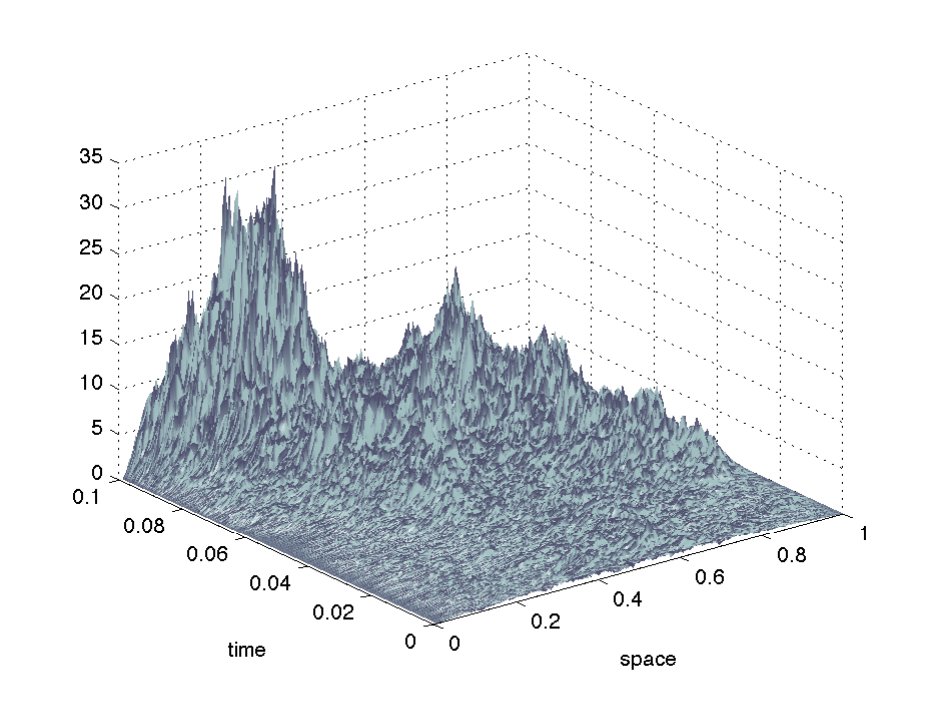}
	\caption{Modest noise ($\lambda=2.0$,  maximum height $\approx 30.53$)}
	\label{fig:lambda2}
\end{figure}
\begin{figure}[h!]\centering
	\includegraphics[width=4in]{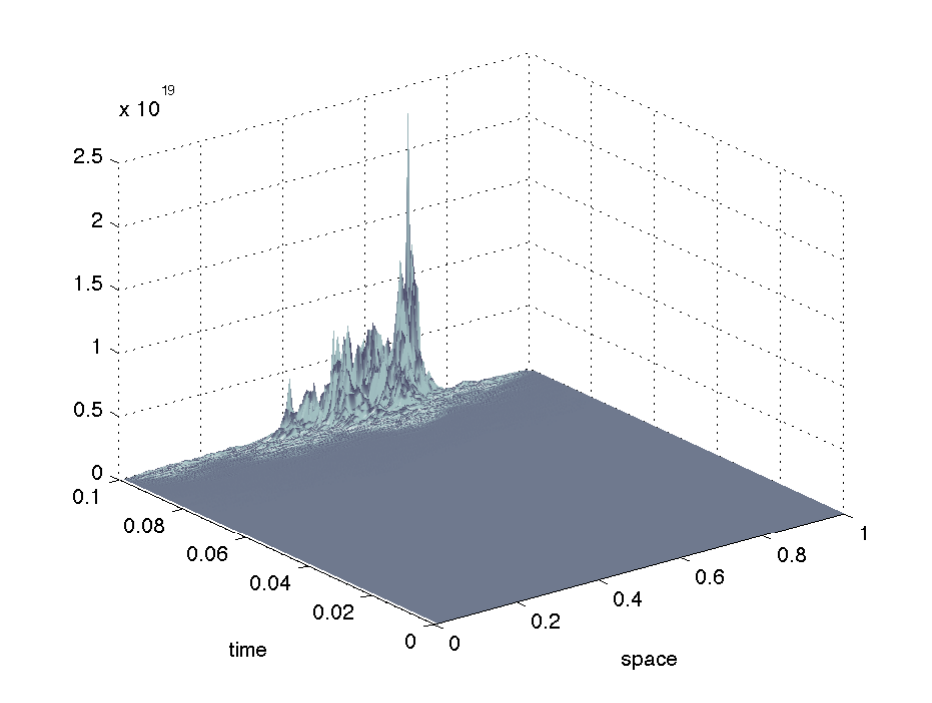}
	\caption{High noise ($\lambda=5.0$,  maximum height $\approx 2.4\times 10^{19}$)}
	\label{fig:lambda5}
\end{figure}

One might notice that, as $\lambda$ increases, the simulated solution 
rapidly develops tall peaks that are distributed over relatively-small 
``islands.'' This general phenomenon is called intermittency; see 
Gibbons and Titi \cite{GibbonsTiti}
for a modern general discussion of intermittency and its 
ramifications, as well as connections to other important topics
[in particular, to turbulence]. Our discussion is motivated in part by
the discussion in Bl\"umich \cite{Blumich} on the quite-different topic of
NMR spectroscopy.

Presently, we study  non-linear noise-excitation phenomena
for stochastic partial differential equations that include the
parabolic Anderson model \eqref{PAM} as a principle example.
Namely, let us choose and fix a length scale $L>0$,
and then consider the stochastic partial differential equation,
\begin{equation}\label{SHE}\left[\begin{split}
	&\frac{\partial}{\partial t} u_t(x) = \frac{\partial^2}{\partial x^2}
		u_t(x) + \lambda\sigma(u_t(x)) \xi&
		\text{for $0< x< L$ and $t>0$},\\
	&u_t(0)=u_t(L)=0&\text{for  $t>0$},
\end{split}\right.\end{equation}
	where $\sigma:\R\to\R$ is a Lipschitz-continuous function, $\xi$ denotes
space-time white noise as before, and the initial function $u_0:[0\,,L]\to\R_+$ is
a non-random bounded continuous function that is non-negative everywhere on $[0\,,L]$
and strictly positive on a set of positive Lebesgue measure in $(0\,,L)$.
For the sake of simplicity, we assume further that 
$\sigma(0)=0$, though some of our work remains valid when $\sigma(0)\neq 0$
as well. 

It is well known that the stochastic heat equation \eqref{SHE} 
has an a.s.-unique continuous solution that has the property that
\begin{equation}\label{eq:moments}
	\sup_{x\in[0,L]}\sup_{t\in[0,T]}\E\left(|u_t(x)|^k\right)<\infty
	\qquad\text{for all $T>0$ and $k\in[2\,,\infty)$}.
\end{equation}
We will be interested in the effect of the level $\lambda$ of the noise
on the [expected] energy $\mathscr{E}_t(\lambda)$ of the solution
at time $t$; the latter quantity is defined as
\begin{equation}
	\mathscr{E}_t(\lambda) := \sqrt{\E\left( \|u_t\|_{L^2[0,L]}^2\right)}\qquad
	(t>0).
\end{equation}
Throughout, we use the following notation:
\begin{equation}\label{ell}
	\ell_\sigma:=\inf_{z\in\R\setminus\{0\}}\left|
	\frac{\sigma(z)}{z}\right|,\qquad
	\lip:=\sup_{z\in\R\setminus\{0\}}\left|
	\frac{\sigma(z)}{z}\right|.
\end{equation}
Clearly, $0\le\ell_\sigma\le\lip$. Moreover, $\lip<\infty$ because 
$\sigma$ is Lipschitz continuous.

The following two theorems contain quantitative descriptions of the non-linear
noise excitability of \eqref{SHE}. These are the main findings of this paper.

\begin{theorem}\label{th:heat}
	For all $t>0$,
	\begin{equation}
		\frac{\ell_\sigma^2 t}{2}\le
		\liminf_{\lambda\to\infty}\frac{1}{\lambda^2}\log
		\mathscr{E}_t(\lambda),\quad
		\limsup_{\lambda\to\infty}\frac{1}{\lambda^4}\log
		\mathscr{E}_t(\lambda) \le 8\lip^4 t.
	\end{equation}
\end{theorem}
In a companion paper \cite{KhKim} we show that $\exp\{c\lambda^4\}$ is
a typical lower bound for the energy of a large number of intermittent complex
systems. In particular, the energy of the solution for the stochastic heat 
equation on $[0\,,L]$ with a \emph{periodic} boundary condition is shown 
to be of sharp exponential order $\exp\{c\lambda^4\}$. 
The following theorem says that we can get the same kind of result 
when we replace a Dirichlet with a Neumann boundary condition, 
provided additionally that the initial profile remains bounded uniformly away from zero.
\begin{theorem}\label{th:heat2}
	Suppose that we replace the Dirichlet boundary condition in 
	\eqref{SHE} by a Neumann boundary condition; that is, we suppose
	that $u$ solves the following stochastic heat equation:
	\begin{equation}\label{SHEN}\left[\begin{split}
	&\frac{\partial}{\partial t} u_t(x) = \frac{\partial^2}{\partial x^2}
		u_t(x) + \lambda\sigma(u_t(x)) \xi&
		\text{for $0< x< L$ and $t>0$},\\
	&\frac{\partial u_t}{\partial x}(0)=\frac{\partial u_t}{\partial x}(L)=0&\text{for  $t>0$}.
	\end{split}\right.\end{equation}
	If, in addition,  we assume that $\inf_{x\in[0,L]}u_0(x)>0$, then 
	for every $t>0$,
	\begin{equation}\label{eq:lambda4:LB}
		\frac{\ell_\sigma^4 t}{8\pi\e}\le
		\liminf_{\lambda\uparrow\infty}
		\frac{1}{\lambda^4}\log\mathscr{E}_t(\lambda)\leq \limsup_{\lambda\uparrow\infty}
		\frac{1}{\lambda^4}\log\mathscr{E}_t(\lambda) \leq \frac{9\lip^4 t}{16}.
	\end{equation}
\end{theorem}
Theorems \ref{th:heat} and \ref{th:heat2} together show that under 
fairly natural regularity conditions [that include $\ell_\sigma>0$],
the energy behaves roughly as $\exp\{c\lambda^4\}$, which is a fastly-growing
function of the level $\lambda$ of the noise. In Section \ref{sec:wave} we document
the somewhat surprising fact
that, by contrast, the stochastic wave equation has typically an energy that
grows merely as $\exp\{c\lambda\}$. In other words, the stochastic wave equation is
typically substantially less noise excitable than the stochastic heat equation.

\section{Proof of Theorem \ref{th:heat}}

As is customary, we begin by writing the solution to the stochastic
heat equation \eqref{SHE} in integral form [also known as the \emph{mild
form}],
\begin{equation}\label{mild}
	u_t(x) = (P_t u_0)(x)
	+ \lambda
	\int_{(0,t)\times(0,L)} p_{t-s}(x\,,y)\sigma(u_s(y))\, \xi(\d s\,\d y),
\end{equation}
where $\{P_t\}_{t\ge 0}$ denotes the semigroup of the Dirichlet Laplacian
on $[0\,,L]$ and $\{p_t\}_{t>0}$ denotes the corresponding heat kernel. That is, 
in particular, $P_0h\equiv h$ for every $h\in L^\infty[0\,,L]$, and
\begin{equation}
	(P_t h)(x) := \int_0^L p_t(x\,,y) h(y)\,\d y
	\qquad\text{for all $t>0$ and $x\in[0\,,L]$}.
\end{equation}
One can expand the heat kernel $p_t(x\,,y)$---the fundamental solution
to the Dirichlet Laplacian on $[0\,,L]$---in 
terms of the eigenfunctions of the Dirichlet Laplacian
as follows:
\begin{equation}
	p_t(x\,,y) := \sum_{n=1}^\infty \phi_n(x)\phi_n(y)
	\e^{-\mu_n t};
\end{equation}
where
\begin{equation}
	\mu_n := \left(\frac{n\pi}{L}\right)^2,\quad
	\phi_n(x) := \left(\frac 2L\right)^{1/2}
	\sin\left(\frac{n\pi x}{L}\right),
\end{equation}
for all $n\ge 1$ and $0\le x\le L$. According to the maximum
principle: (i) $p_t(x\,,y)>0$ for all $t>0$ and $x,y\in(0\,,L)$; and (ii)
$\int_0^Lp_t(x\,,y)\,\d x< 1$ for all $y\in[0\,,L]$ and $t>0$.

Next we recall, briefly, how to one establishes the existence of a mild solution to
\eqref{SHE}.

Let $u^{(0)}_t(x) := u_0(x)$ for all $x\in[0\,,L]$ and $t\ge 0$, and then
define iteratively, for all $k\ge 0$,
\begin{equation}\label{mild}
	u^{(k+1)}_t(x) = (P_t u_0)(x)
	+ \lambda
	\int_{(0,t)\times(0,L)} p_{t-s}(x\,,y)\sigma(u^{(k)}_s(y))\, \xi(\d s\,\d y).
\end{equation}
Then the method of Walsh \cite[Chapter 3]{Walsh}---see Dalang \cite{Dalang}---shows that
$\{u^{(n)}\}_{t\ge 0}$ is locally uniformly Cauchy in $L^2(\P)$
in the sense that for all $T>0$,
\begin{equation}
	\sum_{k=1}^\infty\sup_{x\in[0,L]}\sup_{t\in[0,T]}
	\sqrt{\E\left(|u^{(k+1)}_t(x) - u^{(k)}_t(x)|^2\right)}<\infty.
\end{equation}
It follows fairly easily from this that $u:=\lim_{k\to\infty} u^{(k)}_t(x)$ exists
in $L^2(\P)$ and solves \eqref{mild}. Moreover, one can deduce that
$u$ is unique among all mild solutions that satisfy
\eqref{eq:moments}---see Dalang \cite{Dalang} for the details.

\begin{proof}[Proof of the first bound in Theorem \ref{th:heat}]
	We follow an idea that is classical in the context of PDEs (see Kaplan \cite{Kaplan});
	in the context of SPDEs, a nontrivial adaptation of this idea was used in 
	Bonder and Groisman \cite{BonderGroisman} in order
	to compare the solution of an SPDE to the solution to a one-dimensional
	diffusion in order to describe a certain blow-up phenomenon. 
	In a somewhat loose sense, we follow the same general 
	outline, but need to make a number of modifications along the way.
	
	Recall that $\{\phi_n\}_{n=1}^\infty$ denote the eigenfunctions of
	the Dirichlet Laplacian on $[0\,,L]$ and $\{\mu_n\}_{n=1}^\infty$ are
	the corresponding eigenvalues.
	Since $p_t(x\,,y)=p_t(y\,,x)$,  an appeal to a stochastic Fubini
	theorem \cite[Theorem 2.6, p.\ 296]{Walsh}
	implies the following for all $n\ge 1$ and $t>0$: With probability one,
	\begin{equation}\begin{split}
		&(u_t\,,\phi_n)\\
		&= (u_0\,,P_t\phi_n) + \lambda
			\int_{(0,t)\times(0,L)} (P_{t-s}\phi_n)(y)
			\sigma(u_s(y))\, \xi(\d s\,\d y)\\
		&= \e^{- \mu_n  t}(u_0\,,\phi_n) + \lambda
			\int_{(0,t)\times(0,L)} \e^{-\mu_n(t-s)}\phi_n(y)
			\sigma(u_s(y))\, \xi(\d s\,\d y).
	\end{split}\end{equation}
	Consequently,  the Walsh isometry \cite[Theorem 2.5, p.\ 295]{Walsh} shows us that
	\begin{equation}\begin{split}
		\E\left[ (u_t\,,\phi_n)^2\right]
			&= \e^{- 2\mu_n t}(u_0\,,\phi_n)^2\\
		&\quad + \lambda^2
			\int_0^t\d s\int_0^L\d y\ \e^{-2\mu_n(t-s)}
			|\phi_n(y)|^2\E\left(\left| \sigma(u_s(y))\right|^2\right)\\
		&\ge\e^{- 2\mu_n t}(u_0\,,\phi_n)^2\\
		&\quad + \lambda^2\ell_\sigma^2
			\int_0^t\e^{-2\mu_n(t-s)}\,\d s\int_0^L\d y\ 
			|\phi_n(y)|^2\E\left(\left| u_s(y)\right|^2\right).
	\end{split}\end{equation}
	We may apply the Cauchy--Schwarz inequality, together with an appeal
	to the Fubini theorem, in order to see that
	\begin{equation}
		\int_0^L 
		|\phi_n(y)|^2\E\left(\left| u_s(y)\right|^2\right)\d y
		\ge \frac 1L \E\left[ (u_s\,,\phi_n)^2\right].
	\end{equation}
	Thus we see that, for every fixed $n\ge 1$, the function
	\begin{equation}
		F(t) := \e^{2\mu_n t}\E\left[ (u_t\,,\phi_n)^2\right]
		\qquad(t>0)
	\end{equation}
	satisfies the recursion
	\begin{equation}
		F(t) \ge (u_0\,,\phi_n)^2 + \lambda^2\ell_\sigma^2
		\int_0^tF(s)\,\d s\qquad\text{for all $t>0$}.
	\end{equation}
	Thus the Gronwall's inequality [or just recursion, in this case]
	shows  that 
	\begin{equation}
		F(t) \ge (u_0\,,\phi_n)^2 \exp\left( \lambda^2\ell_\sigma^2 t\right)
		\qquad(t>0).
	\end{equation}
	Equivalently, 
	\begin{equation}
		\E\left[  (u_t\,,\phi_n)^2\right] \ge  
		(u_0\,,\phi_n)^2 \exp\left( \left[\lambda^2\ell_\sigma^2-2\mu_n
		\right] t\right) \qquad(t>0).
	\end{equation}
	Since $\sum_{n=1}^\infty(u_t\,,\phi_n)^2 \le \|u_t\|_{L^2[0,L]}^2$,
	thanks to Bessel's inequality, it remains to prove that the following is strictly positive:
	\begin{equation}
		c_t:=\sum_{n=1}^\infty (u_0\,,\phi_n)^2 \e^{-2\mu_n t}
		\qquad\text{for all $t\ge 0$}.
	\end{equation}
	But $c_t>0$ for all $t\ge 0$ simply because we can write
	$c_t = \|P_tu_0\|_{L^2[0,L]}^2$, and this expression is
	strictly positive, thanks to the maximum principle and the fact that
	$u_0>0$ on a set of positive measure. This concludes the proof of
	the lower bound on the energy in Theorem \ref{th:heat}.
\end{proof}

Our derivation of the upper bound of Theorem \ref{th:heat} requires a 
few elementary preparatory lemmas.

\begin{lemma}\label{lem:ResolventEst}
	For all $\beta,t>0$:
	\begin{equation}
		\sup_{x\in[0,L]} \int_0^t\e^{-\beta(t-s)}\,\d s
		\int_0^L\d y\ [p_{t-s}(x\,,y)]^2
		\le \frac{1}{2\sqrt\beta};
	\end{equation}
	and
	\begin{equation}
		\int_0^t\e^{-\beta(t-s)}\,\d s
		\int_0^L\d x \int_0^L\d y\ [p_{t-s}(x\,,y)]^2
		\le \frac{L}{4\sqrt\beta}.
	\end{equation}
\end{lemma}

\begin{proof}
	Because $|\phi_n(x)|\le (2/L)^{1/2}$ for all $x\in[0\,,L]$ and $n\ge 1$,
	\begin{align}\notag
		\int_0^t\e^{-\beta(t-s)}\,\d s
			\int_0^L\d y\ [p_{t-s}(x\,,y)]^2&=\int_0^t
			\sum_{n=1}^\infty [\phi_n(x)]^2\e^{-2(\beta+\mu_n)(t-s)}\,\d s\\
		&\le \frac1L\sum_{n=1}^\infty\frac{1}{\beta+(n\pi/L)^2}.
		\label{eq:L2:refer}
	\end{align}
	The first bound follows because
	\begin{equation}
		\sum_{n=1}^\infty\frac{1}{\beta+(n\pi/L)^2}\le\int_0^\infty
		\frac{\d x}{\beta+(x\pi/L)^2}=\frac{L}{2\sqrt\beta},
	\end{equation}
	for all $\beta>0$. In order to obtain the second bound, we integrate
	\eqref{eq:L2:refer} $[\d x]$, using the fact that $\|\phi_n\|_{L^2[0,L]}=1$.
\end{proof}

We are now ready to establish the energy upper bound in Theorem \ref{th:heat}.

\begin{proof}[Proof of the second bound in Theorem \ref{th:heat}]
	We appeal to the nonlinear renewal-theoretic technique
	of Foondun and Khoshnevisan \cite{FoondunKh} in order to establish the
	following \emph{a priori} estimate: For every $\delta\in(0\,,1)$,
	\begin{equation}\label{goal1:FK}
		\E\left(|u_t(x)|^m\right) \le \delta^{-m/2}\|u_0\|_{L^\infty[0,L]}^m
		\exp\left(2tm^3\left(\frac{\lambda\lip}{1-\delta}\right)^4 \right),
	\end{equation}
	valid uniformly for all $m\in[2\,,\infty)$,
	$x\in[0\,,L]$, and $t\ge 0$. Once this is established, then
	we set $m=2$ in \eqref{goal1:FK}, and then integrate $[\d x]$,
	in order to deduce the upper bound of the theorem. 
		
	Since $\lip<\infty$ and $\sigma(0)=0$, it follows that
	\begin{equation}
		|\sigma(z)|^2 \le \lip^2|z|^2\qquad
		\text{for all $z\in[0\,,L]$}.
	\end{equation}
	Recall that $u^{(k)}$ denotes the $k$th-stage Picard iteration approximation
	to $u$. According to the Burkholder--Davis--Gundy inequality 
	(see Burkholder \cite{Burkholder},
	Burkholder et al \cite{BDG}, and 
	Burkholder and Gundy \cite{BG};
	see also Foondun and Khoshnevisan \cite{FoondunKh} for an explanation of these
	particular constants),
	\begin{equation}
		\left\| \int_{(0,t)\times(0,L)} Z_s(y)\,\xi(\d s\,\d y)\right\|_m^2
		\le 4m\int_0^t\d s\int_0^L\d y\ \left\|Z_s(y)\right\|_m^2,
	\end{equation}
	for all predictable random fields $Z$, all $t>0$, and $m\in[2\,,\infty)$,
	where $\|Q\|_m:=\{\E(|Q|^m)\}^{1/m}$ for all random variables
	$Q$.
	Therefore, we can deduce from \eqref{mild} that
	\begin{equation}\begin{split}
		&\left\|u^{(k+1)}_t(x)\right\|_m \\
		&\le |(P_tu_0)(x)| + \lambda\sqrt{4m
			\int_0^t\d s\int_0^L\d y\ [p_{t-s}(x\,,y)]^2
			\left\|\sigma(u^{(k)}_s(y))\right\|_m^2}\\
		&\le\|u_0\|_{L^\infty[0,L]}+ \lambda\sqrt{4m\lip^2\int_0^t\d s
			\int_0^L\d y\ [p_{t-s}(x\,,y)]^2
			\left\| u^{(k)}_s(y)\right\|_m^2}.
	\end{split}\end{equation}
	Define
	\begin{equation}
		N_k(\beta) := \sup_{t\ge 0}\sup_{x\in[0,L]}\left(\e^{-\beta t}
		\left\|u_t^{(k)}(x)\right\|_m^2\right)^{1/2}\qquad(k\ge 0,\, \beta>0).
	\end{equation}
	The preceding and Lemma \ref{lem:ResolventEst} together show that
	\begin{equation}
		\int_0^t\d s\int_0^L\d y\ [p_{t-s}(x\,,y)]^2
		\left\| u^{(k)}_s(y)\right\|_m^2 \le
		\left[N_k(\beta) \right]^2\frac{\e^{\beta t}}{2\sqrt\beta}.
	\end{equation}
	Therefore, we combine our efforts in order to deduce the bound
	\begin{equation}
		\left\| u^{(k+1)}_t(x) \right\|_m \le \|u_0\|_{L^\infty[0,L]}
		+\sqrt{2m}\,\lambda\lip\e^{\beta t/2}\beta^{-1/4} N_k(\beta).
	\end{equation}
	
	Since the right-hand side is independent of $x\in[0\,,L]$ and depends on $t$ only
	through the term $\exp(\beta t/2)$, we can divide both sides by $\exp(\beta t/2)$
	and optimize over $t$ and $x$ in order to deduce the recursive inequality,
	\begin{equation}
		N_{k+1}(\beta) \le\|u_0\|_{L^\infty[0,L]}
		+ \frac{\lambda\lip\sqrt{2m}}{\beta^{1/4}}  N_k(\beta).
	\end{equation}
	So far, $\beta$ was an auxiliary positive parameter.
	We now fix it at the value
	\begin{equation}
		\beta_*:= \frac{4m^2(\lambda\lip)^4}{(1-\delta)^4},
	\end{equation}
	where $\delta\in(0\,,1)$ is arbitrary but fixed.
	For this particular choice,
	\begin{equation}
		N_{k+1}(\beta_*) \le\|u_0\|_{L^\infty[0,L]}
		+  (1-\delta) N_k(\beta_*)\qquad(k\ge 0).
	\end{equation}
	Because $N_0(\beta_*)=\|u_0\|_{L^\infty[0,L]}$, the preceding
	iteration yields
	\begin{equation}
		\sup_{k\ge 0}N_k(\beta_*) \le \|u_0\|_{L^\infty[0,L]}\sum_{j=0}^\infty(1-\delta)^j
		=\delta^{-1}\|u_0\|_{L^\infty[0,L]}.
	\end{equation}
	This and Fatou's lemma together imply that
	\begin{equation}
		\sup_{t\ge 0}\sup_{x\in[0,L]}\left(\e^{-\beta_* t}
		\|u_t(x)\|_m^2\right)^{1/2} \le \delta^{-1}\|u_0\|_{L^\infty[0,L]},
	\end{equation}
	which, in turn, yields \eqref{goal1:FK}.
\end{proof}

\section{Proof of Theorem \ref{th:heat2}}
Let us now consider the following stochastic heat equation 
on the interval $[0\,,L]$, subject to a Neumann boundary condition:
\begin{equation}\label{SHEN}\left[\begin{split}
	&\frac{\partial}{\partial t} u_t(x) = \frac{\partial^2}{\partial x^2}
		u_t(x) + \lambda\sigma(u_t(x)) \xi&
		\text{for $0< x< L$ and $t>0$},\\
	&\frac{\partial u_t}{\partial x}(0)=\frac{\partial u_t}{\partial x}(L)=0&\text{for  $t>0$},
\end{split}\right.\end{equation}
where $\xi$ denotes space-time white noise as before, and 
the assumptions on $\sigma$ and $u_0(x)$ are also the same 
as for the Dirichlet problem \eqref{SHE}. 
For the sake of simplicity, we assume further that 
$\sigma(0)=0$, though some of our work remains valid when $\sigma(0)\neq 0$
as well. We may assume that
\begin{equation}\label{ell:positive}
	\ell_\sigma>0.
\end{equation}
Otherwise, the more interesting part of Theorem \ref{th:heat2},
that is the lower bound on the energy, has no content.
[The proof of the upper bound will not require \eqref{ell:positive}.]
Also, let us define
\begin{equation}
	\varepsilon := \inf_{x\in[0,L]}u_0(x).
\end{equation}
According to the hypotheses of Theorem \ref{th:heat2}, $\varepsilon>0$.

As we did for the Dirichlet problem, we begin by writing the 
solution to the stochastic heat equation \eqref{SHEN} in integral form. Namely, we
write
\begin{equation}\label{mildn}
	u_t(x) = (P_t u_0)(x)
	+ \lambda
	\int_{(0,t)\times(0,L)} p_{t-s}(x\,,y)\sigma(u_s(y))\, \xi(\d s\,\d y),
\end{equation}
where $\{P_t\}_{t\ge 0}$ denotes the semigroup of the Neumann Laplacian
on $[0\,,L]$ and $\{p_t\}_{t>0}$ denotes the corresponding heat kernel. That is, $P_0h\equiv h$ for every $h\in L^\infty[0\,,L]$, and
\begin{equation}
	(P_t h)(x) := \int_0^L p_t(x\,,y) h(y)\,\d y
	\qquad\text{for all $t>0$ and $x\in[0\,,L]$}.
\end{equation}
It is well known that we can write
the heat kernel for $\{P_t\}_{t\ge 0}$ in the following form:
\begin{equation}\label{eq:images}
	p_t(x\,,y):=\sum_{n=-\infty}^{\infty} 
	\left[\Gamma_t(x-y-2nL)+\Gamma_t(x+y-2nL)\right],
\end{equation}
where $\Gamma$ denotes the fundamental solution to the heat equation in $\R$;
that is,
\begin{equation}
	\Gamma_t(z):=\frac{1}{\sqrt{4\pi t}} \exp\left( -\frac{z^2}{4t}\right)
	\qquad(t>0,z\in\R).
\end{equation}
[This requires an application of the method of images.]

The existence and uniqueness of a solution of \eqref{SHEN} is shown in Walsh \cite[Chapter 3]{Walsh}. The following establishes the lower bound of Theorem \ref{th:heat2}. 

\begin{proposition}[An energy inequality]\label{pr:energy}
	For all $t>0$, 
	there exists a finite and positive constant
	$K$---independent of $\lambda$---such that
	\begin{equation}
		\mathscr{E}_t(\lambda)\ge \frac{1}{K}
		\exp\left( \frac{(\ell_\sigma\lambda)^4 t}{8\pi\e}\right),
	\end{equation}
	simultaneously for every $\lambda\ge 1$.
\end{proposition}

\begin{proof}[Proof of Proposition \ref{pr:energy}]
	Owing to \eqref{eq:images},
	the Neumann heat kernel is conservative; that is,
	$\int_0^L p_t(x\,,y)\,\d y=1$.
	Therefore, we apply the Walsh isometry to find that for all fixed 
	$\varepsilon,t>0$ and $x\in[0\,,L]$,
	\begin{align}\notag
		\E\left( \left| u_t(x) \right|^2\right)
			&\ge \varepsilon^2 + \lambda^2\int_0^t\d s\int_0^L\d y\
			[p_{t-s}(x\,,y)]^2\E\left(\left| \sigma(u_s(y))\right|^2\right)\\
			\notag
		&\ge \varepsilon^2 + (\lambda\ell_\sigma)^2\int_0^t\d s\int_0^L\d y\
			[p_{t-s}(x\,,y)]^2\E\left(\left| u_s(y)\right|^2\right).
	\end{align} 
	
	The preceding is a recursive, nonlinear renewal-type, inequality for the function 
	$(t\,,x)\mapsto \E(|u_t(x)|^2)$. 
	Even though this renewal inequation is the lower bound's counterpart to
	\eqref{eq:E:wave}, we cannot solve this inequation using the renewal-theoretic
	ideas of Foondun and Khoshnevisan \cite{FoondunKh}. The reason for this 
	is that the method of \cite{FoondunKh} works for large values of the time
	variable $t$; whereas here we have $t$ fixed, but $\lambda$ large. Instead,
	we appeal to old ideas of localization of multiple integrals [i.e., classical
	large deviations] that are probably due to P.-S. Laplace, though we could not
	find explicit references. Namely, we expand the self-referential inequality
	\eqref{eq:RE:LB} and observe that most of the contribution to the resulting
	multiple integrals occur in a small portion of the space of integration.
	
	Now we carry out our program by first expanding, to one term, our
	renewal inequality as follows:
	\begin{align}\notag
		&\E\left( \left| u_t(x) \right|^2\right)\\
		&\ge \varepsilon^2 + \varepsilon^2(\lambda\ell_\sigma)^2\int_0^t\d s\int_0^L\d y\
			[p_{t-s}(x\,,y)]^2\\\notag
		&\quad + (\lambda\ell_\sigma)^4\int_0^t\d s\int_0^L\d y\
			[p_{t-s}(x\,,y)]^2\int_0^s \d r\int_0^L\d z\
			[p_{s-r}(y\,,z)]^2\E\left(\left| u_r(z)\right|^2\right),
	\end{align}
	and proceed. In order to see better what is happening, let us introduce a family of
	linear operators $\mathcal{P}$ as follows: For all $t>0$, $x\in[0\,,L]$,
	and all Borel-measurable functions $\psi:(0\,,\infty)\times[0\,,L]\to\R_+$,
	\begin{equation}
		(\mathcal{P}\psi)(t\,,x) := \int_0^t \d s\int_0^L\d y\ [p_{t-s}(x\,,y)]^2\psi(s\,,y).
	\end{equation}
	Also define $\mathbf{1}(t\,,x) := 1$ for all $t>0$ and $x\in[0\,,L]$.
	Then our recursion can be written compactly as follows:
	\begin{equation}\label{eq:LB}
		\E\left( \left| u_t(x) \right|^2\right)
		\ge \varepsilon^2 \sum_{k=0}^\infty
		(\lambda\ell_\sigma)^{2k}(\mathcal{P}^k\mathbf{1})(t\,,x);
	\end{equation}
	where $\mathcal{P}^0:=\mathbf{1}$ and 
	$\mathcal{P}^{k+1}:=\mathcal{P}^k\mathcal{P}$
	for all $k\ge 0$. We integrate \eqref{eq:LB} $[\d x]$ in order to obtain
	the following key lower bound:
	\begin{equation}\label{eq:L2:LB}
		\left|\mathscr{E}_t(\lambda)\right|^2
		\ge \varepsilon^2\sum_{k=0}^\infty(\lambda\ell_\sigma)^{2k}
		\int_0^L(\mathcal{P}^k\mathbf{1})(t\,,x)\,\d x.
	\end{equation}
	We examine the preceding sum, term by term, using induction.
	
	The first term in the sum is trivial; viz.,
	\begin{equation}\label{P0}
		\int_0^L(\mathcal{P}^0\mathbf{1})(t\,,x)\,\d x = L.
	\end{equation}
	Therefore, let us consider, as a first step, the second term in the sum,
	which is straight forward:
	\begin{equation}\label{Phi}
		\int_0^L (\mathcal{P}^1\mathbf{1})(t\,,x)\,\d x =
		\int_0^t\d s\int_0^L\d x\int_0^L\d y\ [p_{t-s}(x\,,y)]^2
		 := \Phi(t).
	\end{equation}
	Before we estimate the third term of the infinite sum in \eqref{eq:L2:LB}, let
	us observe that
	\begin{align}
		&\int_0^L (\mathcal{P}^2\mathbf{1})(t\,,x)\,\d x\\\notag
		&\qquad= \int_0^t\d s\int_0^L\d x\int_0^L\d y\ [p_{t-s}(x\,,y)]^2
			\int_0^s\d r\int_0^L\d z\ [p_{s-r}(y\,,z)]^2.
	\end{align}
	Therefore, we can bound that third term, from below, as follows:
	\begin{align}
		&\int_0^L (\mathcal{P}^2\mathbf{1})(t\,,x)\,\d x\\\notag
		&\qquad\ge \int_{t/2}^t\d s\int_0^L\d x\int_0^L\d y\ [p_{t-s}(x\,,y)]^2
			\int_{0}^s\d r\int_0^L\d z\ [p_{s-r}(y\,,z)]^2\\\notag
		&\qquad \ge \int_{t/2}^t\d s\int_0^L\d x\int_0^L\d y\ [p_{t-s}(x\,,y)]^2
			\int_{s-t/2}^{s}\d r\int_0^L\d z\ [p_{s-r}(y\,,z)]^2\\\notag
		&\qquad\ge \int_{t/2}^t\d s\int_0^L\d x\int_0^L\d y\ [p_{t-s}(x\,,y)]^2
			\int_0^{t/2}\d r\int_0^L\d z\ [p_r(y\,,z)]^2\\\notag
		&\qquad= \int_0^{t/2}\d s\int_0^{t/2}\d r\int_0^L\d x\int_0^L\d y\
			[p_s(x\,,y)]^2 \int_0^L\d z\  [p_r(y\,,z)]^2.
	\end{align}
		It is not important that we have a $t/2$ and
    in the bounds of our integrals on the second line; $t/\e$ 
    would have worked
	equally well for us. The key point is that most of the contribution to our multiple integral
	comes about where $s\approx t$ and $r\approx s$ in the second line.
	
	Now, let us recall that $p_r(y\,,z)$ can be
	thought of as the transition function for Brownian 
	motion in $(0\,,L)$, reflected upon reaching the boundary $\{0\,,L\}$
	\cite{PortStone}. The
	Markov property of reflected Brownian motion implies the semigroup property
	of $\{P_t\}_{t\ge 0}$ and in particular that
	$\int_0^L[p_r(y\,,z)]^2\,\d z=p_{2r}(y\,,y)$. And by symmetry,
	$\int_0^L[p_s(x\,,y)]^2\,\d x=p_{2s}(y\,,y)$, as well. Consequently,
	we obtain 
	\begin{equation}
		\int_0^L (\mathcal{P}^2\mathbf{1})(t\,,x)\,\d x
		\ge \int_0^L\left|\int_0^{t/2}\ [p_{2s}(y\,,y)]^2 \,\d s\right|^2\,\d y.
	\end{equation}
	Once we understand this, we can see how to bound the remaining
	terms in the infinite sum in \eqref{eq:L2:LB} as well. In fact, induction shows that
	\begin{equation}
		\int_0^L (\mathcal{P}^{k+1}\mathbf{1})(t\,,x)\,\d x
		\ge \int_0^L\left|\int_0^{t/(k+1)}\ [p_{2s}(y\,,y)]^2 \,\d s\right|^{k+1}\,\d y,
	\end{equation}
	for all $k\ge 1$. Once again we stress that our multiple integrals are
	bounded from below by a very small portion of the region of integration.
	
	Finally, we apply Jensen's inequality in order to see that
	\begin{equation}
		\int_0^L (\mathcal{P}^{k+1}\mathbf{1})(t\,,x)\,\d x
		\ge L^{-k}\left|\int_0^L\d y\int_0^{t/(k+1)}\d s\ [p_{2s}(y\,,y)]^2  
		\right|^{k+1},
	\end{equation}
	for all $k\ge 1$. Recall \eqref{Phi} in order to see that
	\begin{equation}
		\int_0^L (\mathcal{P}^k\mathbf{1})(t\,,x)\,\d x
		\ge L\left[\frac{\Phi(t/k)}{L}\right]^k,
	\end{equation}
	for all $k\ge 1$ and $t>0$. This and \eqref{eq:L2:LB} together yield
	\begin{equation}\label{eq:L2:LB1}\begin{split}
		\left|\mathscr{E}_t(\lambda)\right|^2
			&\ge \varepsilon^2 L \cdot  \sum_{k=1}^\infty
			\left[\frac{(\lambda\ell_\sigma)^2\Phi(t/k)}{L}\right]^k.
	\end{split}\end{equation}
	Finally, we estimate the function $\Phi$ from below.
	
	For all $\tau>0$,
	\begin{equation}\begin{split}
		\Phi(\tau) &= \int_0^\tau\d s\int_0^L\d x\int_0^L\d y\
			[p_s(x\,,y)]^2\\
		&=\int_0^\tau \d s\int_0^L \d x\ p_{2s}(x\,,x).
	\end{split}\end{equation}
	Since $p_{2s}(x\,,x)\ge \Gamma_{2s}(0)=
	(8\pi s)^{-1/2}$ for all $s>0$---see \eqref{eq:images}---the preceding yields the 
	pointwise bound
	\begin{equation}
		\Phi(\tau)\geq\int_0^\tau \frac{L}{\sqrt{8\pi s}}\, \d s
		=L\sqrt{\frac{\tau}{2\pi}}\qquad(\tau>0).
	\end{equation}
	Thus, we obtain the following: 
	\begin{equation}\label{eq:L2:LB2}\begin{split}
		\left|\mathscr{E}_t(\lambda)\right|^2
			&\ge \varepsilon^2 L \cdot \sum_{k=1}^\infty
			\left[\frac{(\lambda\ell_\sigma)^2\sqrt{t}}{\sqrt{2\pi k}}\right]^k\\
		&\ge \varepsilon^2 L \cdot \sum_{j=1}^\infty
			\left[\frac{(\lambda\ell_\sigma)^2\sqrt{t}}{\sqrt{4\pi j}}\right]^{2j}\\
		&\ge \varepsilon^2 L \cdot \sum_{j=1}^\infty
			\left[\frac{(\lambda\ell_\sigma)^4 t}{4\pi\e}\right]^j \frac{1}{j!},
	\end{split}\end{equation}
	thanks to the elementary bound, $(j/\e)^j\le j!$, valid for all integers $j\ge 1$.
	The proposition follows from the preceding and the Taylor series expansion
	of the exponential function.
\end{proof}

Before we prove the upper bound in Theorem \ref{th:heat2}, 
let us state a simple result regarding the Neumann heat kernel.

\begin{lemma}\label{lem:kernel}
	For every $\varepsilon>0$ there exists a positive
	and finite constant $K:=K_{\varepsilon,L}$ such that
	\begin{equation}
		\sup_{t\geq 0}\sup_{0\leq x\leq L} \int_0^t \e^{-\beta s}
		p_{2s}(x\,,x) \, \d s \le \frac{3+\varepsilon}{\sqrt{8\beta}}
		\qquad\text{for all $\beta\ge K$}.
	\end{equation}
\end{lemma}

\begin{proof}
	In accord with \eqref{eq:images},
	for every $s>0$ and $x\in[0\,,L]$,
	\begin{equation}\begin{split}
		p_{2s}(x\,,x)&=\sum_{n=-\infty}^\infty \left[
			\Gamma_{2s}(2nL)+\Gamma_{2s}(2x-2nL)\right]\\
		&\le 3\Gamma_{2s}(0)+2\sum_{n=1}^\infty\Gamma_{2s}(2nL)+
			\sum_{\substack{n\in\Z:\\|nL-x|\ge 1}}\Gamma_{2s}(2|nL-x|)\\
		&\le 3\Gamma_{2s}(0) + C(L)\\
		&=\frac{3}{\sqrt{8\pi s}}+C(L),
	\end{split}\end{equation}
	uniformly for all $s>0$ and $x\in[0\,,L]$, where $C(L)$ is a positive
	and finite constant that depends only on $L$. The numerical bound of $3$
	[in front of $\Gamma_{2s}(0)$] accounts for the
	fact that, depending on the value of $x$, there can be at most two
	choices of $n\in\Z$ such that $|nL-x|<1$. We integrate the preceding
	in order to find that
	\begin{equation}
		\int_0^\infty\e^{-\beta s}p_{2s}(x\,,x)\,\d s \le\frac{3}{\sqrt{8\beta}} +
		\frac{C(L)}{\beta},
	\end{equation}
	which has the desired effect.
\end{proof}

We need only to prove the upper bound in Theorem \ref{th:heat2};
the corresponding lower bound has already been established.

\begin{proof}[Proof of the second bound in Theorem \ref{th:heat2}]
	According to the Walsh isometry,
	for all $t>0$
	\begin{align}
		&\E\left(|u_t(x)|^2\right)\\
			\notag
		& =|(P_tu_0)(x)|^2+\lambda^2\int_0^t\d s\int_0^L\d y\
			[p_{t-s}(x\,,y)]^2\E\left(\left| \sigma(u_s(y))\right|^2\right)\\
				\notag
			&\le \|u_0\|_{L^\infty[0,L]}^2  + (\lambda\lip)^2\int_0^t\d s\int_0^L\d y\
				[p_{t-s}(x\,,y)]^2\E\left(\left| u_s(y)\right|^2\right).
	\end{align}
	We solve this inequality by applying the method of Foondun
	and Khoshnevisan \cite{FoondunKh};  namely, let us define, for all $\beta>0$, 
	\begin{equation}
		\mathscr{N}(\beta) := 
		\sup_{t\ge 0}\sup_{0\leq x\leq L}\left[\e^{-\beta t}\E\left(|u_t(x)|^2
		\right) \right].
	\end{equation}
	The preceding and Lemma \ref{lem:kernel} together
	imply that for all $\varepsilon>0$ there exists $K:=K_{\varepsilon,L}$
	such that for all $\beta\ge K$,
	\begin{align}\notag
		\mathscr{N}(\beta) &\leq \|u_0\|_{L^\infty[0,L]}^2 + 
			\mathscr{N}(\beta) (\lambda\lip)^2 
			\sup_{t\geq 0}\sup_{0\leq x\leq L} 
			\int_0^t \e^{-\beta s} [p_{2s}(x\,,x)]^2 \, \d s\\
		&\leq \|u_0\|_{L^\infty[0,L]}^2
			+\frac{(3+\varepsilon)(\lambda\lip)^2}{\sqrt{8\beta}} \mathscr{N}(\beta).
	\end{align}	
	We let $0< \delta<1$ and choose $\beta$ as
	\begin{equation}
		\beta^*:=\frac{(3+\varepsilon)^2(\lambda\lip)^4}{8(1-\delta)^2},
	\end{equation}
	in order to deduce the inequality
	\begin{equation}
		\mathscr{N}(\beta^*) \leq \frac{1}{\delta}\|u_0\|_{L^\infty[0,L]}^2,
	\end{equation}
	valid as long as $\lambda$ is large enough
	to ensure that $\beta^*\geq K$. In this way we find that
	\begin{equation}\label{eq:upper}
		\E\left(|u_t(x)|^2\right) \leq \frac{1}{\delta}\|u_0\|_{L^\infty[0,L]}^2
		\exp\left( \frac{(3+\varepsilon)^2(\lambda\lip)^4 t}{8(1-\delta)^2}\right).
	\end{equation}
	We integrate both sides of \eqref{eq:upper} [$\d x$]
	in order to obtain the upper bound, after some relabeling.
\end{proof}

\section{A stochastic wave equation}\label{sec:wave}
In this section we consider the non-linear stochastic wave equation
\begin{equation}\label{SWE}
	\frac{\partial^2}{\partial t^2} w_t(x) = \frac{\partial^2}{\partial x^2}
	w_t(x) + \lambda\sigma(w_t(x)) \xi\qquad
	\text{for $x\in\R$ and $t>0$},
\end{equation}
subject to non-random initial function
$w_0(x) \equiv 0$ and non-random non-negative
initial velocity $v_0\in L^1(\R)\cap L^2(\R)$ such that
$\|v_0\|_{L^2(\R)}>0$. It is well known (see Dalang \cite{Dalang}
for the general theory,
as well as Dalang and Mueller \cite{DalangMueller},
for the existence of an $L^2$-valued solution)
that there exists  a unique continuous
solution $w$ to \eqref{SWE} that satisfies the moment conditions,
\begin{equation}
	\sup_{t\in[0,T]}\sup_{x\in\R}\E\left(\left| w_t(x)\right|^k\right)+
	\sup_{t\in[0,T]}\mathscr{E}_t(\lambda)<\infty,
\end{equation}
for every $k\in[2\,,\infty)$ and $T>0$.

The main result of this section are the following bounds on
the energy $\mathscr{E}_t(\lambda):=\{\E(\|u_t\|_{L^2(\R)}^2)\}^{1/2}$
of the solution to \eqref{SWE}.

\begin{theorem}\label{th:wave}
	For every $t>0$,
	\begin{equation}
		\frac{\ell_\sigma t}{4\sqrt{8}}\le\liminf_{\lambda\to\infty}\frac{1}{\lambda}\log
		\mathscr{E}_t(\lambda)\le
		\limsup_{\lambda\to\infty}
		\frac{1}{\lambda}\log\mathscr{E}_t(\lambda) \le 
		\frac{\lip t}{\sqrt 8}.
	\end{equation}
\end{theorem}
Define
\begin{equation}
	H(r) := \frac r2\wedge \frac{r^2}{4}\qquad(r>0).
\end{equation}
The preceding theorem requires the following elementary real-variable
fact. 

\begin{lemma}\label{lem:RealVar}
	Suppose $g$ is a non-negative element of $L^1(\R)\cap L^2(\R)$.
	Then, there exist positive and finite constants $A_1$ and $A_2$---depending
	only on $\|g\|_{L^1(\R)}$ and $\|g\|_{L^2(\R)}$---such that
	\begin{equation}
		A_1H(t)\le
		\int_{-t}^t\d z\int_{-t}^t\d y\ \left( g*\tilde{g}\right)(y-z)
		\le A_2H(t),
	\end{equation}
	for all $t>0$, where
	$\tilde g(x):=g(-x)$ for all $x\in\R$.
\end{lemma} 

\begin{proof}
	Define $h :=g*\tilde g$ for simplicity, and note that: (i) $h\ge 0$;
	(ii) $h\in L^1(\R)\cap L^2(\R)$ with
	$\|h\|_{L^1(\R)}=\|g\|_{L^1(\R)}^2$ and
	$h(0)=\|g\|_{L^2(\R)}^2.$ Furthermore, since $h(x)$ is maximized
	at $x=0$, thanks to well-known facts about continuous,
	positive-definite functions.
	Now we put these facts together in order to see that
	\begin{equation}\begin{split}
		\int_{-t}^t\d y\int_{-t}^t\d z\ h(y-z) &\le 4h(0)t^2\wedge 2t\|h\|_{L^1(\R)}\\
		&\le 16\left(\|g\|_{L^2(\R)}^2 \wedge \|g\|_{L^1(\R)}^2\right)H(t).
	\end{split}\end{equation}
	This proves the upper bound [with an explicit $A_2$].
	
	On the other hand,
	$\int_{-t}^t\d y\int_{-t}^t\d z\ h(y-z) =(4+o(1))t^2h(0)$ as $t\downarrow 0$
	and
	$\int_{-t}^t\d y\int_{-t}^t\d z\ h(y-z) = (2+o(1))t\|h\|_{L^1(\R)}$
	as $t\uparrow\infty$. The lower bound follows from these observations,
	since $h$ is non-negative.
\end{proof}

\begin{proof}[Proof of Theorem \ref{th:wave}]
	The solution to the stochastic wave equation \eqref{SWE} can be
	written in mild form as follows:
	\begin{equation}
		w_t(x) = \tfrac12 W_t(x) + 
		\tfrac12 \lambda\int_{(0,t)\times\R} \mathbf{1}_{[0, t-s]}(|x-y|)
		\sigma(w_s(y))\,\xi(\d s\,\d y),
	\end{equation}
	where
	\begin{equation}
		W_t(x) := \int_{-t}^t v_0(x-y)\,\d y.
	\end{equation}
	Therefore, the Walsh isometry for stochastic integrals assures us that
	\begin{equation}\begin{split}
		&\E\left(|w_t(x)|^2\right)\\
		&= \tfrac14|W_t(x)|^2 + \tfrac14\lambda^2\int_0^t\d s\int_{-\infty}^\infty
			\d y\ \mathbf{1}_{[0,t-s]}(|x-y|)\E\left(|\sigma(w_s(y))|^2\right),
	\end{split}\end{equation}
	whence by Fubini's theorem,
	\begin{equation}\begin{split}
		\left|\mathscr{E}_t(\lambda)\right|^2
			&=\tfrac14\|W_t\|_{L^2(\R)}^2 + \tfrac12
			\lambda^2\int_0^t(t-s)\,\d s\int_{-\infty}^\infty
			\d y\ \E\left(|\sigma(w_s(y))|^2\right)\\
		&\le\tfrac14\|W_t\|_{L^2(\R)}^2 + \tfrac12\lambda^2\lip^2
			\int_0^t(t-s)\,\d s\int_{-\infty}^\infty
			\d y\ \E\left(|w_s(y)|^2\right).
	\end{split}\end{equation}
	
	Since $\|W_t\|_{L^2(\R)}^2 = \int_{-t}^t\d y\int_{-t}^t\d z\
	(v_0*\tilde{v}_0 )(y-z)$, Lemma \ref{lem:RealVar} ensures that
	the squared energy $\left|\mathscr{E}_t(\lambda)\right|^2$ of the solution to the stochastic
	wave equation satisfies the renewal inequality,
	\begin{equation}\label{eq:E:wave}
		\left|\mathscr{E}_t(\lambda)\right|^2
		\le A_2 t^2\|v_0\|_{L^2(\R)}^2 + \tfrac12\lambda^2\lip^2
		\int_0^t(t-s) \left(\mathscr{E}_s(\lambda)\right)^2\,\d s,
	\end{equation}
	for all $t>0$. 
	
	Since
	\begin{equation}
		\sup_{t\ge 0}\left[ t^2\e^{-\beta t}\right]
		= \frac{4}{(\e\beta)^2},
	\end{equation}
	the preceding implies that
	\begin{equation}\begin{split}
		\mathscr{F}(\beta) &\le \frac{4A_2}{(\e\beta)^2} \|v_0\|_{L^2[0,L]}^2 + 
			\tfrac12\lambda^2
			\lip^2\mathscr{F}(\beta)\int_0^t(t-s)\e^{-\beta(t-s)}\,\d s\\
		&\le \frac{4A_2\|v_0\|_{L^2[0,L]}^2}{(\e\beta)^2} + 
			\tfrac12\lambda^2\lip^2\F(\beta)
			\int_0^\infty s\e^{-\beta s}\,\d s\\
		&\le \frac{4A_2\|v_0\|_{L^2[0,L]}^2}{(\e\beta)^2}+ 
			\frac{\lambda^2\lip^2\F(\beta)}{2\beta^2}.
			\label{eq:Fbeta}
	\end{split}\end{equation}
	Let us choose and fix an arbitrary $\delta\in(0, 1)$.
	We can apply an \emph{a priori} method in order to show that
	$\mathscr{E}(\beta_*)<\infty$, where
	\begin{equation}
		\beta_* := \frac{\lambda\lip}{\sqrt{2(1-\delta)}}.
	\end{equation}
	See the derivation of \eqref{goal1:FK}; we omit the details.
	In this way, we can solve \eqref{eq:Fbeta} in order to deduce the following:
	\begin{equation}
		\mathscr{F}(\beta_*) \le\frac{8A_2
		\|v_0\|_{L^2[0,L]}^2}{\delta(\e\lambda\lip)^2}.
	\end{equation}
	Equivalently,
	\begin{equation}
		\left|\mathscr{E}_t(\lambda)\right|^2\le\frac{8A_2\|v_0\|_{L^2[0,L]}^2}{%
		\delta(\e\lambda\lip)^2}
		\exp\left(\frac{\lambda\lip t}{\sqrt{2(1-\delta)}}\right).
	\end{equation}
	This readily yields the upper bound of the theorem.
	
	We prove the corresponding lower bound by observing the 
	following counterpart of \eqref{eq:E:wave}, which holds for
	the same reasons as \eqref{eq:E:wave} does:
	\begin{equation}\label{eq:RE:LB}
		\left|\mathscr{E}_t(\lambda)\right|^2
		\ge A_1H(t)\|v_0\|_{L^2[0,L]}^2 + 
		\tfrac12\lambda^2\ell_\sigma^2
		\int_0^t(t-s) \left|\mathscr{E}_s(\lambda)\right|^2\,\d s.
	\end{equation}
	Now we proceed as we did in the proof of Proposition \ref{pr:energy};
	we expand our renewal inequation as an inequality in terms of an infinite
	sum of multiple integrals of increasingly-high powers. And then show that
	the multiple integrals are large, when $\lambda\gg1$, mainly because
	of the contribution of the integrand in a small region of integration.
	In this way the remainder of the proof is exactly the same as the derivation
	of Proposition \ref{pr:energy}. However, as it turns out, one has to be
	quite careful in order to guess the correct region of integration, as it will
	be very significantly larger than the one in the proof Proposition \ref{pr:energy}.
	
	With the preceding in mind, we begin by writing
	\begin{align}\notag
		\left|\mathscr{E}_t(\lambda)\right|^2
			&\ge A_1H(t)\|v_0\|_{L^2[0,L]}^2 
			+ \tfrac12\lambda^2\ell_\sigma^2
			\int_0^t(t-s) \left(\mathscr{E}_s(\lambda)\right)^2\,\d s\\\notag
		&\ge A_1H(t)\|v_0\|_{L^2[0,L]}^2 
			+ \tfrac12 A_1\lambda^2\ell_\sigma^2\|v_0\|_{L^2[0,L]}^2
			\int_0^t(t-s) H(s)\,\d s\\
		&\hskip.5in+\left(\tfrac12\lambda^2\ell_\sigma^2\right)^2
			\int_0^t\d s\int_0^s\d r\ (t-s)(s-r)\left|
			\mathscr{E}_r(\lambda)\right|^2,
	\end{align}
	etc. In this way, we obtain the following generous lower bound,
	\begin{align}
		\label{range}
		&\left|\mathscr{E}_t(\lambda)\right|^2\\\notag
		&\ge A_1\|v_0\|_{L^2[0,L]}^2\sum_{n=1}^\infty
			\left(\tfrac18\lambda^2\ell_\sigma^2\right)^n
			\int_{t/2}^t\d s_1\int_{s_1/2}^{s_1}\d s_2
			\cdots\int_{s_{n-1}/2}^{s_{n-1}}\d s_n\ \mathscr{S},
	\end{align}
	where $\mathscr{S} :=  (t-s_1)\times(s_1-s_2)\times\cdots
	\times(s_{n-1}-s_n)\times
	H(t),$ over the range of the integral in \eqref{range}. Let us emphasize
	that the $n$th term involves an $n$-fold integral that is integrate on
	a large part of the original region of integration; this is in sharp contrast 
	with the proof of Proposition \ref{pr:energy}, where the effective region
	of integration was extremely small for the $n$th term when $n\gg1$.
	Once the correct region of integration is identified, we can continue
	the argument in the proof of Proposition \ref{pr:energy}.
	Namely, we obtain the following bounds, after we appeal to time reversal:
	\begin{equation}\begin{split}
		\left|\mathscr{E}_t(\lambda)\right|^2
			&\ge A_1\|v_0\|_{L^2[0,L]}^2H(t)\cdot\sum_{n=1}^\infty
			\left(\tfrac18\lambda^2\ell_\sigma^2\right)^n\\
		&\hskip1in\times
			\int_0^{t/2}\d s_1\int_0^{s_1/2}\d s_2
			\cdots\int_0^{s_{n-1}/2}\d s_n\ \prod_{j=1}^n s_j.
	\end{split}\end{equation}
	A change of variables and  induction together yield
	\begin{equation}\begin{split}
		&\int_0^{t/2}\d s_1\int_0^{s_1/2}\d s_2
			\cdots\int_0^{s_{n-1}/2}\d s_n\ \prod_{j=1}^n s_j\\
		&\hskip.8in=4^{-n}\int_0^t\d s_1\int_0^{s_1}\d s_2
			\cdots\int_0^{s_{n-1}}\d s_n\ \prod_{j=1}^n s_j\\
		&\hskip.8in=\frac{t^{2n}}{4^n \cdot (2n)!!},
	\end{split}\end{equation}
	where $(2n)!!:=(2n)\times(2n-2)\times\cdots\times 2$ denotes the usual
	double factorial of $2n$. We require only the simple bound
	$(2n)!!\le (2n)!$ in order to deduce that
	\begin{equation}
		\left|\mathscr{E}_t(\lambda)\right|^2
		\ge A_1\|v_0\|_{L^2[0,L]}^2H(t)\cdot\sum_{n=1}^\infty
		\left(\frac{\lambda\ell_\sigma t}{2\sqrt{8}}\right)^{2n}\frac{1}{(2n)!}.
	\end{equation}
	Since
	\begin{equation}
		\sum_{n=1}^\infty \left(\frac{\lambda\ell_\sigma t}{%
		2\sqrt{8}}\right)^{2n+1}\frac{1}{(2n+1)!}
		\le \sum_{n=1}^\infty \left(\frac{\lambda\ell_\sigma t}{2\sqrt{8}}\right)^{2n}
		\frac{1}{(2n)!},
	\end{equation}
	for $\lambda\ge2\sqrt 8/(t\ell_\sigma)$.
	It follows that
	\begin{equation}
		\left|\mathscr{E}_t(\lambda)\right|^2
		\ge \tfrac12 A_1\|v_0\|_{L^2[0,L]}^2H(t)
		\cdot\sum_{j=2}^\infty
		\left(\frac{\lambda\ell_\sigma t}{2\sqrt{8}}\right)^j\frac{1}{j!},
	\end{equation}
	whenever $\lambda\ge2\sqrt 8/(t\ell_\sigma)$. Because
	\begin{equation}
		\sum_{j=0}^1
		\left(\frac{\lambda\ell_\sigma t}{2\sqrt{8}}\right)^j\frac{1}{j!} 
		= O\left(\lambda\right)\qquad\text{as $\lambda\uparrow\infty$},
	\end{equation}
	the lower bound follows.
\end{proof}

\spacing{.8}
\begin{small}
\bigskip

\noindent\textbf{Davar Khoshnevisan} \&\ \textbf{Kunwoo Kim}\\
\noindent Department of Mathematics, University of Utah,
		Salt Lake City, UT 84112-0090 \\
		
\noindent\emph{Emails} \& \emph{URLs}:\\
	\indent\textcolor{purple}{\texttt{davar@math.utah.edu}}\hfill
		\textcolor{purple}{\url{http://www.math.utah.edu/~davar/}}
	\indent\textcolor{purple}{\texttt{kkim@math.utah.edu}}\hfill
		\textcolor{purple}{\url{http://www.math.utah.edu/~kkim/}}
\end{small}

\end{document}